\documentclass[11pt, a4paper, reqno, english]{smfart}

\usepackage[margin=4cm]{geometry}
\usepackage{amsfonts, amsthm, amsmath, amscd, amssymb}

\usepackage[draft=false]{hyperref}
\usepackage[T1]{fontenc}
\usepackage[all]{xy}
\usepackage{graphicx}

\renewcommand\AA{\mathbb{A}}
\newcommand\QQ{\mathbb{Q}}
\newcommand\RR{\mathbb{R}}
\newcommand\TT{\mathcal{T}}
\newcommand\ZZ{\mathbb{Z}}
\newcommand\ZZp{\ZZ_{>0}}
\newcommand\ZZnz{\ZZ_{\ne 0}}
\newcommand\RRp{\RR_{>0}}

\newcommand\PP{\mathbb{P}}
\newcommand\xx{\mathbf{x}}

\newcommand\R{\mathcal{R}}
\newcommand\dd{\,\mathrm{d}}
\newcommand\Ptwo{{\PP^2}}
\newcommand\Pthree{{\PP^3}}
\newcommand{\sbase}[6]{\ee^{({#1},{#2},{#3},{#4},{#5},{#6})}}
\newcommand{\congr}[3]{{#1} \equiv {#2}\ (\mathrm{mod}\ {#3})}
\newcommand{\cp}[2]{{\gcd(#1,#2)=1}}
\newcommand{\ncp}[2]{{\gcd(#1,#2) > 1}}
\newcommand{\Atwo}{{\mathbf A}_2}
\newcommand{\Dfive}{{\mathbf D}_5}
\newcommand{\Esix}{{\mathbf E}_6}
\newcommand{\tS}{{\widetilde S}}
\renewcommand{\le}{\leqslant}
\renewcommand{\ge}{\geqslant}

\renewcommand{\geq}{\geqslant}
\newcommand{\ee}{\boldsymbol{\eta}}
\renewcommand{\aa}{\boldsymbol{\alpha}}
\newcommand{\bb}{\boldsymbol{\beta}}
\newcommand{\gag}{\boldsymbol{\gamma}}
\newcommand{\ex}[1]{*+<10pt>[o][F]{#1}}
\newcommand\rto{\dashrightarrow}
\newcommand\e{\eta}
\newcommand\al{\alpha}
\newcommand\phis{\phi^*}
\newcommand\Ga{\mathbb{G}_\mathrm{a}}
\newcommand\gothm{\mathfrak{m}}
\newcommand\midd{\,\Big|\,}
\renewcommand\rho{\varrho}
\renewcommand\theta{\vartheta}
\newcommand\intdist[1]{\|#1\|}
\newcommand\ve{\varepsilon}

\DeclareMathOperator{\Cox}{Cox}
\DeclareMathOperator{\vol}{vol}
\DeclareMathOperator{\Spec}{Spec}
\DeclareMathOperator{\meas}{meas}

\newtheorem*{theorem}{Theorem}
\newtheorem{lemma}{Lemma}
\theoremstyle{remark}
\newtheorem*{ack}{Acknowledgements}
\numberwithin{equation}{section}
\theoremstyle{definition}

\begin{document}

\title{Manin's conjecture for a cubic surface with $\Dfive$ singularity}

\author{T.~D. Browning}
\address{School of Mathematics\\
University of Bristol\\ Bristol\\ BS8 1TW\\ United Kingdom}
\email{t.d.browning@bristol.ac.uk}

\author{U. Derenthal} 
\address{Institut f\"ur Mathematik\\ 
Universit\"at Z\"urich\\
Winterthurerstrasse 190\\
8057 Z\"urich\\ 
Switzerland}
\email{ulrich.derenthal@math.unizh.ch}

\date{October 22, 2008}

\begin{abstract}
  The Manin conjecture is established for a split singular cubic
  surface in $\Pthree$, with
  singularity type $\Dfive$.
\end{abstract}


\maketitle
\tableofcontents

\section{Introduction}

Let $S \subset \Pthree$ be the cubic surface defined by 
\begin{equation}
  \label{eq:surface}
    x_3x_0^2+x_0x_2^2+x_2x_1^2 = 0.
\end{equation}
Then $S$ is a singular del Pezzo surface with a unique singularity
$(0:0:0:1)$ of type $\Dfive$ and three lines, each of which 
is defined over $\QQ$.

\begin{figure}[ht]
  \centering
  \includegraphics[width=13cm]{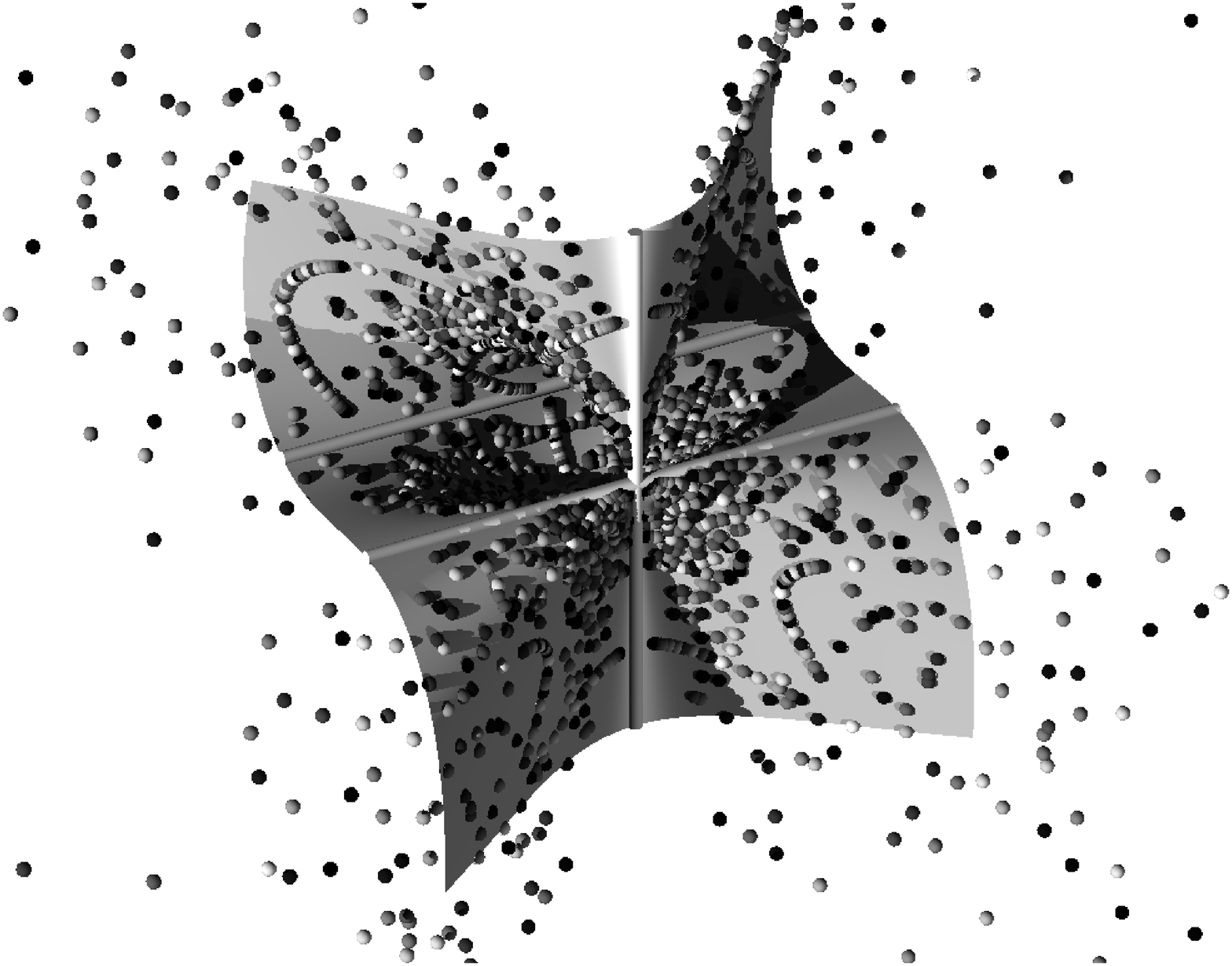}
  \caption{Points of height $\le 100$ on the $\Dfive$ cubic surface.}
  \label{fig:d5}
\end{figure}

Let $U$ be the Zariski open subset formed by deleting the lines 
from $S$. Our principal object of study in this paper is the cardinality
\[N_{U,H}(B) = \#\{\xx \in U(\QQ) \mid H(\xx) \le B\},\] for any $B\ge 1$.
Here $H$ is the usual height on $\Pthree$, in which $H(\xx)$ is defined as
$\max\{|x_0|,\dots,|x_3|\}$, provided that the point $\xx \in \Pthree(\QQ)$ is
represented by integral coordinates $(x_0,\ldots,x_3)$ that are relatively
coprime. In Figure~\ref{fig:d5} we have plotted an affine model of $S$,
together with all of the rational points of low height that it contains.  The
following is our principal result.

\begin{theorem}
  We have \[N_{U,H}(B) = c_{S,H}B(\log B)^6 + O\left(B(\log B)^5(\log \log B)\right),\] where the leading
  constant is \[c_{S,H} = \frac{1}{230400} \cdot \omega_\infty \cdot \prod_p
  \omega_p\] with
  \begin{align*}
    \omega_p &=\left(1-\frac 1 p\right)^7\left(1+\frac 7 p+\frac
    1{p^2}\right),\\
        \omega_\infty &= \int_{|x_0|,|x_1|,|x_2|,|x_0^{-2}(x_0x_2^2+x_2x_1^2)| \le 1,\
          x_2 \ge 0} x_0^{-2} \dd x_0 \dd x_1 \dd x_2.
      \end{align*}
\end{theorem}

It is straightforward to check that the surface $S$ is neither toric nor an
equivariant compactification of $\Ga^2$. Thus this result does not follow from
the work of Tschinkel and his collaborators \cite{MR1620682, clt}.  Our
theorem confirms the conjecture of Manin \cite{MR89m:11060} since the Picard
group of the minimal desingularisation $\tS$ of the split del Pezzo surface
$S$ has rank $7$.  Furthermore, the leading constant $c_{S,H}$ coincides with
Peyre's prediction \cite{MR1340296}. To check this we begin by observing that
\[\alpha(\tS) = \frac{\alpha(S_0)}{\#W(\Dfive)} =
\frac{1/120}{1920} = \frac{1}{230400},\] by \cite[Theorem~4]{MR2318651} and
\cite[Theorem~1.3]{math.AG/0703202}, where $S_0$ is a 
split smooth cubic surface and $\#W(\Dfive)$ is the order of the Weyl group of the
root system $\Dfive$. Next one easily verifies that the constant
$\omega_\infty$ in the theorem is the real density, which is computed by
writing $x_3$ as a function of $x_0,x_1,x_2$ and using the Leray form
$x_0^{-2} \dd x_0 \dd x_1 \dd x_2$. Finally, it is straightforward to compute
the $p$-adic densities as being equal to $\omega_p$.

Our work is the latest in a sequence of attacks upon the Manin
conjecture for del Pezzo surfaces, a comprehensive survey of which can be
found in \cite{gauss}. 
A number of authors have established the conjecture for the surface 
$$
x_1x_2x_3+x_0^3=0,
$$ 
which has singularity
type $3\Atwo$.  The sharpest unconditional result available is due to 
la Bret\`eche \cite{MR2000b:11074}. 
Furthermore, in joint work with  la Bret\`eche \cite{MR2332351}, the
authors have recently resolved the conjecture for the surface
$$
x_1x_2^2+x_2x_0^2+x_3^3 = 0,
$$ 
which has singularity type $\Esix$.  Our main result signifies only the third
example of a cubic surface for which the Manin conjecture has been resolved.

The proof of the theorem draws upon the expanding store of technical machinery
that has been developed to study the growth rate of rational points on
singular del Pezzo surfaces.  In particular, we will take advantage of the
estimates involving exponential sums that featured in \cite{MR2332351}. In the
latter setting these tools were required to get an asymptotic formula for the
relevant counting function with error term of the shape $O(B^{1-\delta})$.
However, in their present form, they are not even enough to establish an
asymptotic formula in the $\Dfive$ setting.  Instead we will need to revisit
the proofs of these results in order to sharpen the estimates to an extent
that they can be used to establish the theorem.  In addition to these refined
estimates, we will often be in a position to abbreviate our argument by taking
advantage of \cite{manin_auxiliary}, where several useful auxiliary results
are framed in a more general context.

In keeping with current thinking on the arithmetic of split del Pezzo
surfaces, the proof of our theorem relies on passing to a universal
torsor, which in the present setting is an open subset of the
hypersurface 
\begin{equation}
  \label{eq:torsor}
  \eta_2\eta_6^2\alpha_2 +
  \eta_4\eta_5^2\eta_7^3\eta_8 + \eta_3\alpha_1^2=0,
\end{equation}
embedded in $\AA^{10} \cong \Spec\QQ[\e_1,\ldots,\e_8,\al_1,\al_2]$.
Furthermore, as with most proofs of the Manin conjecture for singular del
Pezzo surfaces of low degree, the shape of the cone of effective divisors of
the corresponding minimal desingularisation plays an important role in our
work.  For the surfaces treated in \cite{MR2320172}, \cite{MR2332351},
\cite{math.NT/0604193}, the fact that the effective cone is simplicial
streamlines the proofs considerably. For the surface studied in
\cite{arXiv:0710.1560}, this was not the case, but it was nonetheless possible
to exploit the fact that the dual of the effective cone is the difference of
two simplicial cones.  For the cubic surface \eqref{eq:surface},
the dual of the effective cone is again the difference of two simplicial
cones. However, we choose to ignore this fact and rely on a more
general strategy instead.

\begin{ack}
  While working on this paper the first author was supported by EPSRC grant
  number \texttt{EP/E053262/1}. The second author was partially supported by
  a Feodor Lynen Research Fellowship of the Alexander von Humboldt
  Foundation. The authors are grateful to the referee for a number of useful
  comments that have improved the exposition of this paper.
\end{ack}

\section{Arithmetic functions and exponential sums}

Define the multiplicative arithmetic functions 
\[
\phis(q)=\prod_{p\mid q}\left(1-\frac{1}{p}\right), \quad 
g(q) = \sum_{d \mid q} d^{-1/2},\quad h_k(q) = 2^{\omega(q)}g(q)^k,
\]  
for any $k \in \ZZp$, where $\omega(q)$ denotes the number of distinct prime
factors of $q$. These functions will feature quite heavily in our work and we
will need to know the average order of the latter.

\begin{lemma}\label{lem:sum_h_k}
  For any $k\in \ZZp$ we have \[\sum_{q \le Q} h_k(q) \ll_k Q \log Q.\]
\end{lemma}

\begin{proof}
Let $k\in \ZZp$ be given and let $\ve > 0$. Then we have 
\begin{align*}
\sum_{q\le Q}h_k(q)
&=\sum_{q\le Q}2^{\omega(q)}
\sum_{d_1,\ldots,d_k\mid q}(d_1\cdots d_k)^{-1/2}\\
&\ll_\ve
\sum_{d_1,\ldots,d_k=1}^{\infty}(d_1\cdots d_k)^{\ve-1/2}\sum_{u\le Q/[d_1,\ldots,d_k]}2^{\omega(u)}\\
&\ll_\ve Q\log Q
\sum_{d_1,\ldots,d_k=1}^{\infty}\frac{(d_1\cdots d_k)^{\ve-1/2}}{[d_1,\ldots,d_k]},
\end{align*}
where $[a,b]$ denotes the least common multiple of $a,b\in \ZZp$. We
easily check  that the final sum is absolutely convergent by
considering the corresponding Euler product, which has local factors
of the shape $1+O_\ve(p^{\ve-3/2})$. 
\end{proof}

Given integers $a,b,q$, with $q>0$, we will be led to consider the
quadratic exponential sum
\begin{equation}\label{eq:exp_sum_quadratic}
  S_q(a,b)=\sum_{v=1}^q e_q(av^2+bv).
\end{equation}
Our study of this should be compared with the corresponding sum studied in
\cite[Eq. (3.1)]{MR2332351}, involving instead a cubic phase $av^3+bv^2$. In
\cite[Lemma~4]{MR2332351} an upper bound of the shape
$O_\ve(\gcd(q,b)q^{1/2+\ve})$ is established for the cubic sum. The following
result shows that we can do better in the quadratic setting.

\begin{lemma}\label{lem:lemma4}
  For any $a,b \in \ZZ$ with $\gcd(q,a,b)=1$, we have
  \begin{equation*}
    S_q(a,b) \ll \gcd(q,a)^{1/2}q^{1/2}.
  \end{equation*}
\end{lemma}

\begin{proof}
 Writing $w=v+x$ in the second step, we find that
   \begin{equation*}
     \begin{split}
       |S_q(a,b)|^2 &= \sum_{v,w=1}^q e_q(a(v^2-w^2)+b(v-w))\\
       & = \sum_{v,x=1}^q e_q(-a(2vx+x^2)-bx)\\
       & = \sum_{x=1}^q e_q(-ax^2-bx) \sum_{v=1}^q e_q(-2avx)
     \end{split}
   \end{equation*}
   The inner sum is $q$ if $q \mid 2ax$ and $0$ otherwise. Let $h=\gcd(a,q)$
   and write $q=hq'$, $a=ha'$ with $\cp{a'}{q'}$. Then
   \begin{equation*}
     |S_q(a,b)|^2 = q \sum_{\substack{x=1\\q' \mid 2x}}^q e_q(-ax^2-bx) \le 2qh,
   \end{equation*}
   and the result follows.
 \end{proof}

Our next results concern the function $\psi(t)=\{t\}-1/2$, where $\{t\}$ is
the fractional part of $t \in \RR$.  The following estimate improves upon
\cite[Lemma~5]{MR2320172}.

\begin{lemma}\label{lem:sum_mod_q}
  For any $t\in \RR$, $b \in \ZZ$, $q \in \ZZp$ with $\cp b q$, we
  have \[\sum_{\substack{x=1\\\cp x q}}^q \psi\left(\frac{t-bx^2}{q}\right)
  \ll h_1(q) \log(q+1)q^{1/2}.\]
\end{lemma}

\begin{proof}
  Let $S(q)$ denote the sum that is to be estimated. By M\"obius inversion it
  follows that
  \begin{equation*}
    \begin{split}
      S(q)&= \sum_{n \mid q} \mu(n) \sum_{0 \le x' <q/n}
      \psi\left(\frac{t/n-bnx'^2}{q/n}\right)\\
      &= \sum_{\substack{n \mid q\\m=\gcd(n,q/n)}} \mu(n)m \sum_{0 \le x' <
            \frac{q}{mn}} \psi\left(\frac{t/(mn) - bnx'^2/m}{q/(mn)}\right).
    \end{split}
  \end{equation*}
  We claim that
  \begin{equation}\label{eq:sum_psi_square}
    \sum_{0 \le x < q} \psi\left(\frac{t-bx^2}{q}\right) \ll 
    g(q) \log(q+1)q^{1/2},
  \end{equation}
for any $t\in \RR$, $b \in \ZZ$ and $q \in \ZZp$ with $\cp b q$. Under
this assumption, it therefore follows that 
  \begin{equation*}
    \begin{split}
S(q)  &\ll \sum_{\substack{n \mid q\\m=\gcd(n,q/n)}}|\mu(n)|m g(q) \log(q+1)
      (q/(mn))^{1/2}\\
      &= g(q) \log(q+1) q^{1/2} \sum_{n\mid q}
      \frac{|\mu(n)|\gcd(n,q/n)^{1/2}}{n^{1/2}}\\
      &\ll 2^{\omega(q)}g(q)\log(q+1) q^{1/2}.
    \end{split}
  \end{equation*}
  This is satisfactory for the lemma, since $h_1(q) = 2^{\omega(q)}g(q)$.

  To establish \eqref{eq:sum_psi_square} we follow the proof of \cite[Lemma~4]{MR2320172}, finding that
  \begin{equation*}
    \sum_{0 \le x < q} \psi\left(\frac{t-bx^2}{q}\right) \ll 1 +
    \sum_{m\mid q} \sum_{\substack{1 \le \ell' < q/m\\\cp{\ell'}{q/m}}}
    \frac{|T(q,m,\ell')|}{q \intdist{\ell'm/q}},
  \end{equation*}
  where
  \begin{equation*}
    T(q,m,\ell') = \sum_{0 \le x < q} e_{q/m}(\ell'bx^2).
  \end{equation*}
  Rather than applying Weyl's inequality as in \cite[Lemma~4]{MR2320172}, we
  simply break into $m$ residue classes modulo $q/m$ and apply Lemma
  \ref{lem:lemma4} to deduce that
  \begin{equation*}                         
    T(q,m,\ell') \ll m(q/m)^{1/2} = (mq)^{1/2}.
  \end{equation*}

  Now
  \begin{equation*}
    \sum_{1 \le \ell' < r} \intdist{\ell'/r}^{-1} \ll r\sum_{1 \le
      \ell'<r}\ell'^{-1} \ll r \log(r+1),
  \end{equation*}
  for any $r \in \ZZp$. Hence
  \begin{equation*}
    \begin{split}
 \sum_{0 \le x < q} \psi\left(\frac{t-bx^2}{q}\right)
      &\ll 1 + \log(q+1)q^{1/2} \sum_{m \mid q} m^{-1/2}\\
      &\ll g(q)\log(q+1)q^{1/2},
    \end{split}
  \end{equation*}
which thereby concludes the proof of \eqref{eq:sum_psi_square}.
\end{proof}

For positive integers $a,b$, we define the function
\begin{equation}\label{eq:def_f}
 f_{a,b}(n) =
 \begin{cases}
  \phis(n)/\phis(\gcd(n,a)), & \mbox{if $\cp n b$},\\
  0, & \mbox{if $\ncp n b$}.
 \end{cases}
\end{equation}
We combine Lemma~\ref{lem:sum_mod_q} with the proof of
\cite[Lemma~1]{arXiv:0710.1560} to obtain the following result. 

 \begin{lemma}\label{lem:sum_mod_q_2}
   Let $0 \le t_1 < t_2$ and $\cp \alpha q$. We have
   \begin{multline*}
     \sum_{\substack{1 \le \rho \le q\\\cp{\rho}{q}}} \sum_{\substack{t_1 < n
         \le t_2\\\congr n {\alpha\rho^2} q}} f_{a,b}(n) =
     (t_2-t_1)\cdot\phis(bq) \prod_{p\nmid abq}\left(1-\frac{1}{p^2}\right)\\
     + O\left(2^{\omega(b)}\log(t_2+2) h_1(q) \log(q+1)q^{1/2}\right).
   \end{multline*}
 \end{lemma}

 \begin{proof}
   In the proof of
   \cite[Lemma~1]{arXiv:0710.1560}, \[\Sigma=\sum_{\substack{d^{-1}t_1 < m \le
       d^{-1}t_2\\\congr{md}{\alpha\rho^2} q}} 1\] is estimated as
   $(t_2-t_1)/(dq)+O(1)$, for given $d$ coprime to $q$. Using
   \cite[Lemma~7]{MR2332351}, we make this precise
   as \[\Sigma=\frac{t_2-t_1}{dq}+\psi\left(\frac{d^{-1}t_1-\overline d \alpha
       \rho^2}{q}\right) - \psi\left(\frac{d^{-1}t_2-\overline d
       \alpha\rho^2}{q}\right),\] where $\overline d$ is chosen such that
   $\congr{d \overline d} 1 q$. Our task is to compute \[\sum_{\substack{1 \le \rho
       \le q\\\cp \rho q}} \sum_{\substack{1 \le d \le t_2\\\cp d q}} (f_{a,b}
   * \mu)(d) \Sigma.\]

   For the main term, we may extend the summation over $d$ to all positive
   integers, since \[\sum_{1 \le \rho \le q} \frac{t_2-t_1}{q}
   \sum_{\substack{d > t_2\\\cp d q}} \frac{(f_{a,b}*\mu)(d)}{d} \ll t_2
   \sum_{d > t_2} \frac{\gcd(b,d)|\mu(d)|}{d^2} \ll 2^{\omega(b)}.\] As in
   \cite[Lemma~1]{arXiv:0710.1560}, we see that the sum over $d\in \ZZp$ is
   $c_0(t_2-t_1)/q$, with \[c_0 = \prod_{\substack{p\mid b\\p \nmid q}}
   \left(1-\frac{1}{p}\right) \prod_{p\nmid abq}\left(1-\frac{1}{p^2}\right).
   \] Summing this over $\rho$, we get $c_0
   \phis(q) (t_2-t_1)$. It is easy to see that $c_0\phis(q)$ agrees with the
   leading constant in the statement of the lemma.

   For the error term, we exchange the summations over $d$ and $\rho$. Applying
   Lemma~\ref{lem:sum_mod_q}, we obtain the contribution \[
   \begin{split}
     &\ll F(q)\sum_{\substack{1 \le d \le t_2\\\cp d q}}
     |(f_{a,b}*\mu)(d)| \ll 2^{\omega(b)}\log(t_2+2)F(q),
   \end{split}
   \]
   with $F(q)=h_1(q) \log(q+1)q^{1/2}$.
   This completes the proof of the lemma.
 \end{proof}

Given $b,c,q \in \ZZ$ such that $q>0$ and a real-valued function $f$
defined on an interval $I\subset \RR$, let 
\[
S_I(f,q)=\sum_{x \in \ZZ \cap I}
\sum_{\substack{y=1\\\congr{y^2}{bx}{q}\\\cp y q}}^q
\psi\left(\frac{f(x)-cy}q\right).
\] 
It is interesting to compare this sum with the sort of sums that
featured in our corresponding investigation of the $\Esix$ cubic surface. 
The sole difference between \cite[Eq. (4.1)]{MR2332351} and $S_I(f,q)$
is that the argument involves $(f(x)-cxy)/q$, rather than $(f(x)-cy)/q$.

We will be interested in studying $S_I(f,q)$ when 
$f \in C^1(I;\lambda_0)$. Here, if $I=[t_1,t_2]$ and $\lambda_0\ge 1$,
then $C^1(I;\lambda_0)$ is defined to be the set of real-valued
differentiable functions $f$, such that $f'$ is monotonic
  and of constant sign on $(t_1,t_2)$, with $|f(t_2)-f(t_1)|+1 \le
  \lambda_0$.
It will be convenient to define 
\[
\gothm(I)=\meas(I)+2.
\] 
We will need a version of \cite[Lemma~10]{MR2332351}, in which the
factor $q^\ve \gothm(I)^\ve$
is made more explicit. This is achieved in the following result.

\begin{lemma}\label{lem:lemma10}
  Let $X = q \gothm(I)$. Assume
  $\cp{bc}q$ and $f \in C^1(I;\lambda_0)$. For any $\ve > 0$, we have
  $$
  S_I(f,q) \ll \left(h_2(q) q^{1/2} +\tau(q)^2\frac{\gothm(I)}{q}
    +\frac{h_1(q)}{\log X}
    \frac{\lambda_0^{1/2}\gothm(I)^{1/2}}{q^{1/4}}\right)(\log X)^2,
  $$
  where $\tau(n)=\sum_{d|n} 1$ is the divisor function.
\end{lemma}

In comparing this with \cite[Lemma~10]{MR2332351}, one sees that 
the first and third term in both results share the same approximate order of
magnitude. However, the middle term is improved from $1/q^{1/3}$ to
$1/q$. This saving is crucial in our work. It arises from the 
fact that the current set-up leads us to estimate the quadratic exponential
sums \eqref{eq:exp_sum_quadratic} with $a=0$, rather than the
corresponding cubic sums with phase $av^3+bv^2$ and $b=0$. In the
former case we are dealing with linear exponential sums, for which we
have very good control, and in the  latter case 
we only have the bound $O_\ve (q^{2/3+\ve})$ available.

\begin{proof}[Proof of Lemma \ref{lem:lemma10}]
  Let $\eta(\al;q)=\#\{1 \le n \le q \mid \congr{n^2}{\al}{q}\}$. Replacing
  the bound $\eta(\al;q) \ll_\ve q^\ve$ by $\eta(\al;q) \le
  2^{\omega(q)+1}$ in the application of Vaaler's
  trigonometric formula in the proof of \cite[Lemma~10]{MR2332351}, we obtain \[S_I(f,q) \ll
  \frac{2^{\omega(q)} \gothm(I)}{H} + \sum_{h=1}^H \frac 1 h |T_I(f,q;h)|\]
  for any $H \ge 1$, where
  \begin{equation*}
    T_I(f,q;h) = \sum_{x \in \ZZ \cap I}
    \sum_{\substack{y=1\\\congr{y^2}{bx}{q}\\\cp y q}}^q e_q(hf(x)-chy).
  \end{equation*}
  As in \cite[Lemma~10]{MR2332351}, we rewrite this as
  \begin{equation*}
    T_I(f,q;h) = \frac 1 q \sum_{k=1}^q A_I(q;-k,h,f)B(q;h,k),
  \end{equation*}
  with
  \begin{equation*}
    A_I(q;-k,h,f) = \sum_{x \in \ZZ \cap I} e_q(-kx+hf(x))
  \end{equation*}
  and
  \begin{equation*}
    B(q;h,k) = \sum_{u=1}^q
    \sum_{\substack{v=1\\\congr{v^2}{bu}{q}\\\cp v q}}^q e_q(ku-chv) =
    \sum_{\substack{v=1\\\cp v q}}^q e_q(\overline b k v^2 - chv),
  \end{equation*}
  where $\overline b$ is the multiplicative inverse of $b$ modulo $q$.
  Since $\gcd(q,\overline b k, ch) = \gcd(q,k,h)$, we have (with $h=dh',\
  k=dk',\ q=dq'$)
  \begin{equation*}
    T_I(f,q;h) =
    \sum_{d|h,q} \frac{1}{dq'} \sum_{\substack{-q'/2 < k'\le
        q'/2\\\gcd(k',h',q')=1}} A_I(q';-k',h',f)B(dq';dh',dk').
  \end{equation*}
  
  Write each $v$ modulo $q$ uniquely as $v=y+q'z$ with $1 \le y \le q'$ and $1
  \le z \le d$. Then
  \begin{equation*}
    B(dq';dh',dk') = \sum_{y=1}^{q'}\sum_{\substack{z=1\\\cp{y+q'z}{dq'}}}^d
    e_{q'}(\overline b k'y^2-ch'y)=f(d,q')B(q';h',k')
  \end{equation*}
  with $f(d,q') \le d$, just as in \cite[Lemma~10]{MR2332351}.  Therefore,
  \begin{equation*}
    T_I(f,q;h) \ll \sum_{d \mid h,q} \frac{1}{q'} \sum_{\substack{-q'/2 < k'
        \le q'/2\\\gcd(k',h',q')=1}} |A_I(q';-k',h',f)||B(q';h',k')|.
  \end{equation*}

  For the contribution from the case $k'=0$, note that $\cp{h'}{q'}$. We have
  $A_I(q';0,h',f) \ll \gothm(I)$  trivially, and
  \begin{equation*}
    B(q';h',0) = \sum_{\substack{v=1\\\cp v{q'}}}^{q'} e_{q'}(-ch'v) =
    \sum_{d\mid q'} \mu(d) \sum_{v=1}^{q'/d} e_{q'/d}(-ch'v).
  \end{equation*}
  The inner sum is $q'/d$ if $(q'/d) \mid ch'$ (which is possible only in the
  case $q'/d=1$ since $\gcd(q,c)=\cp {q'}{h'}$) and $0$ otherwise. Thus
  $B(q';h',0) = \mu(q')$, whence the total contribution to $T_I(f,q;h)$ from
  the case $k'=0$ is
  \begin{equation*}
    \ll \sum_{d\mid h,q} \frac 1{q'} \gothm(I) |\mu(q')| \ll
    \frac{\gothm(I)}{q} \sigma(\gcd(h,q)),
  \end{equation*}
  where $\sigma(n) = \sum_{d\mid n} d$ is the sum of divisors function.
  
  For the total contribution to $T_I(f,q;h)$ from the case $k' \ne 0$, we note
  that
  \begin{equation*}
    A_I(q';-k',h',f) \ll \frac{1}{|k'|}(q'+h'\lambda_0) =
    \frac{q'}{|k'|}(1+h\lambda_0/q),
  \end{equation*}
  by \cite[Lemma~5]{MR2332351} for $f \in C^1(I;\lambda_0)$. Also
  \begin{equation*}
    B(q';h',k') = \sum_{\substack{v=1\\\cp{v}{q'}}}^{q'} e_{q'}(\overline b k'
    v^2-ch'v) = \sum_{e\mid q'} \mu(e) \sum_{v=1}^{q''} e_{q''}(\overline b
    k'ev^2 - ch'v),
  \end{equation*}
  where $q'=eq''$. By Lemma~\ref{lem:lemma4},
  \begin{equation*}
    \begin{split}
      |B(q';h',k')| &\ll \sum_{e\mid q'} |\mu(e)| q''^{1/2} \gcd(q'',\overline
      bk'e)^{1/2}\\
      &\le \sum_{e\mid q'} |\mu(e)|\frac{q'^{1/2}}{e^{1/2}} \gcd(q',k')^{1/2}
      \gcd(q'',e)^{1/2}\\
      &\le 2^{\omega(q')} \gcd(q',k')^{1/2} q'^{1/2}.
    \end{split}
  \end{equation*}
  The contribution from the case $k'\ne 0$ is therefore
  \begin{equation*}
    \begin{split}
      &\ll \sum_{d\mid h,q} \frac d q \sum_{\substack{1 \le k' \le
          q'/2\\\gcd(k',h',q')=1}} \frac{q'}{k'} (1+h\lambda_0/q)
      2^{\omega(q')} \gcd(q',k')^{1/2} q'^{1/2}\\
      &\ll (1+h\lambda_0/q) 2^{\omega(q)} q^{1/2} \sum_{d \mid h,q}
      \frac{1}{d^{1/2}}\sum_{k' \le q'/2} \frac{\gcd(q',k')^{1/2}}{k'}\\
      &\ll (1+h\lambda_0/q) h_2(q) \log(q+1) q^{1/2}.
    \end{split}
  \end{equation*}

  Plugging the contribution from $k' = 0$ and $k' \ne 0$ to $T_I(f,q;h)$ into
  $S_I(f,q)$, we deduce, for any $H \ge 1$, that $S_I(f,q)$ is
  $$
     \ll{} \frac{2^{\omega(q)} \gothm(I)}{H}
      + \sum_{h=1}^H \frac 1 h \left( (1+h\lambda_0/q) h_2(q)
        \log(q+1) q^{1/2} + \frac{\gothm(I)}{q} 
        \sigma(\gcd(h,q)) \right).
  $$
Observing that
  \begin{equation*}
    \sum_{h \le H} \frac{\sigma(\gcd(h,q))}{h} = \sum_{d \mid q} \sigma(d)
    \sum_{\substack{h \le H \\ d \mid h}} \frac{1}{h} \ll (\log H) \sum_{d \mid q}
    \frac{\sigma(d)}{d} \ll (\log H) \tau(q)^2, 
  \end{equation*}
we therefore deduce that
\begin{multline*}
S_I(f,q) \ll \frac{2^{\omega(q)} \gothm(I)}{H} + \frac{\tau(q)^2 \gothm(I)
        (\log H)}{q}+ \log(q+1)h_2(q) q^{1/2} (\log H)\\
      + \log(q+1)h_2(q) \lambda_0 H/q^{1/2}.
    \end{multline*}

  Let
  \begin{equation*}
    H = \frac{q^{1/4} \gothm(I)^{1/2}}{\lambda_0^{1/2}\log(q+1)^{1/2}g(q)}.
  \end{equation*}
  If $H \ge 1$, we may use this $H$ in the estimate above, together with $q+1
  \le X$, in order to obtain the lemma.  If $H < 1$, so that $q^{1/4}
  \gothm(I)^{1/2} < \lambda_0^{1/2}(\log q)^{1/2} g(q)$, we deduce from the
  trivial estimate $S_I(f,q) \ll 2^{\omega(q)} \gothm(I)$ that the lemma holds
  in this case too.
\end{proof}

\section{The universal torsor}

Let $S$ be the $\Dfive$ cubic surface \eqref{eq:surface}, let
$U\subset S$ be the open subset formed by deleting the lines from
$S$ and let $\tS$ be the minimal desingularisation of $S$.
In this section we will establish an explicit bijection between $U(\QQ)$
and the integral points on the universal torsor above $\tS$,
subject to a number of coprimality conditions.
For this we will follow the strategy explained in \cite{math.NT/0604193}.

To establish the bijection we will introduce new variables
$\eta_1,\ldots,\eta_8$ and $\al_1,\al_2$. It will be convenient to henceforth
write
$$
\ee=(\e_1,\dots,\e_6),\quad 
\ee'=(\e_1,\dots,\e_8),
\quad \aa =
(\al_1,\al_2)
$$
and 
$$
\sbase{k_1}{k_2}{k_3}{k_4}{k_5}{k_6} = \prod_{i=1}^6 \e_i^{k_i},  
$$
for any $(k_1,\dots,k_6) \in \QQ^6$.

Let us recall some information concerning the geometry of $S$ from
\cite[Section~8]{math.AG/0604194}.
Blowing up the singularity $(0:0:0:1)$ on $S$ results in the exceptional
divisors $E_1, \dots, E_5$ in a $\Dfive$-configuration on the minimal
desingularisation $\pi: \tS \to S$. Let $E_6, E_7, E_8$ resp.\ $A_1, A_2$ on
$\tS$ be the strict transforms under $\pi$ of the three lines
$E_6''=\{x_0=x_1=0\}$, $E_7''=\{x_0=x_2=0\}$, $E_8''=\{x_2=x_3=0\}$ resp.\ the
curves $A_1''=\{x_1=x_0x_3+x_2^2=0\}$ and $A_2'' = \{x_3=x_0x_2+x_1^2=0\}$ on
$S$. The  extended
Dynkin diagram in Figure~\ref{fig:dynkin} is the dual graph of the
configuration of the curves $E_1, \dots, E_8, A_1, A_2$ on $\tS$.

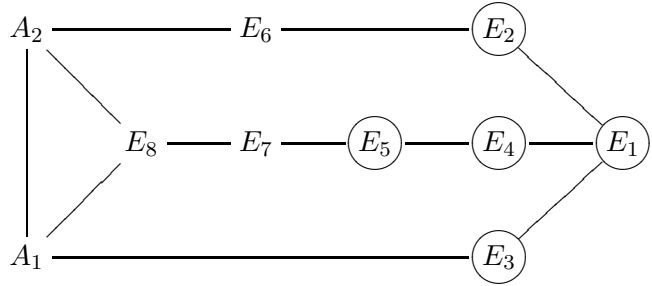
\begin{figure}[ht]
  \centering
  \[\xymatrix{A_2 \ar@{-}[rr] \ar@{-}[dr] \ar@{-}[dd] & & E_6  \ar@{-}[rr]& & \ex{E_2} \ar@{-}[dr]\\
      & E_8 \ar@{-}[r] & E_7 \ar@{-}[r] & \ex{E_5} \ar@{-}[r] & \ex{E_4} \ar@{-}[r] & \ex{E_1}\\
      A_1 \ar@{-}[rrrr] \ar@{-}[ur] & & & & \ex{E_3} \ar@{-}[ur]}\]
  \caption{Configuration of curves on $\tS$.}
  \label{fig:dynkin}
\end{figure}

By \cite[Section~8]{math.AG/0604194}, non-zero global sections $\e_1, \dots,
\e_8, \al_1, \al_2$ corresponding to $E_1, \dots, E_8, A_1, A_2$ form 
a generating set of the Cox ring of $\tS$. The ideal of relations in $\Cox(\tS)$
is generated by $\e_2\e_6^2\al_2+\e_4\e_5^2\e_7^3\e_8+\e_3\al_1^2$. We 
express the sections $\pi^*(x_i)$,  for $0\le i\le 3$, of the 
anticanonical class $-K_\tS$ in terms of the generators of $\Cox(\tS)$
as follows: 
\begin{equation*}
  (\pi^*(x_0), \dots, \pi^*(x_3)) = (\sbase 4 3 2 3 2 2\e_7 , 
  \sbase 3 2 2 2 1 1 \al_1, \sbase 2 1 1 2 2 0 \e_7^2 \e_8, \e_8\al_2).
\end{equation*}

The general strategy of \cite{math.NT/0604193} suggests that $U(\QQ)$
should be parametrised by certain integral 
points on the variety $\Spec(\Cox(\tS))$. This is confirmed in the 
the following result.

\begin{lemma}\label{lem:bijection}
  We have
  \[N_{U,H}(B) = \#\TT(B),\] where $\TT(B)$ is the set of $(\ee',\aa) \in
  \ZZp^7 \times \ZZnz \times \ZZ^2$ such that \eqref{eq:torsor} holds,
  with
  \begin{equation}
    \label{eq:height}
    \max\{|\sbase 4 3 2 3 2 2\e_7| , |\sbase 3 2 2 2 1 1 \al_1|, 
    |\sbase 2 1 1 2 2 0 \e_7^2 \e_8|, |\e_8\al_2|\} \le B
  \end{equation}
  and
  \begin{align}
    \label{eq:cpa2} &\cp{\al_2}{\e_1\e_2\e_7},\\
    \label{eq:cpa1} &\cp{\al_1}{\e_1\e_4\e_5},\\
    \label{eq:cpe8} &\cp{\e_8}{\e_1\e_2\e_3\e_4\e_5\e_6},\\
    \label{eq:cpe7} &\cp{\e_7}{\e_1\e_2\e_3\e_4\e_6},\\
    \label{eq:cpe} &\text{coprimality between $\e_1, \dots, \e_6$ as in
      Figure~\ref{fig:dynkin}.}
  \end{align}
\end{lemma}

The coprimality conditions in \eqref{eq:cpe} are achieved by taking $\e_i$ and
$\e_j$ to be coprime if and only if the divisors $E_i$ and $E_j$ are not
adjacent in the diagram.  The reader is invited to consider the correspondence
between
\begin{itemize}
\item the variables of the parametrisation and the generators of $\Cox(\tS)$,
\item the torsor equation~(\ref{eq:torsor}) and the relation in $\Cox(\tS)$,
\item the height conditions~(\ref{eq:height}) and the expressions of
  $\pi^*(x_i)$ in terms of the generators of $\Cox(\tS)$,
\item the coprimality conditions~(\ref{eq:cpa2})--~(\ref{eq:cpe}) and the
  configuration of the curves associated to the generators of $\Cox(\tS)$
  encoded in Figure~\ref{fig:dynkin}.
\end{itemize}

The proof of Lemma~\ref{lem:bijection} is elementary, but modelled according
to the geometry of $S$. The following additional geometric information is
relevant. Contracting $E_6, E_2, E_1, E_3, E_7, E_5$ in this order leads to a
map $\phi_1: \tS \to \Ptwo$ that is the blow-up of six points in the
projective plane. We may choose $\phi_1(E_4), \phi_1(A_1), \phi_1(E_8)$ as the
coordinate lines in $\Ptwo = \{(\e_4': \al_1': \e_8')\}$. Then $\phi_1(A_2)$
is the quadric $\e_4'\e_8'+\al_1'^2=0$. The morphisms $\phi_1, \pi$ and the
projection
\begin{equation*}
  \begin{array}[h]{cccc}
    \phi_2: & S & \rto & \Ptwo,\\
    & \xx & \mapsto & (x_0:x_1:x_2)
  \end{array}
\end{equation*}
from the singularity $(0:0:0:1)$, form a commutative diagram of rational maps
between $\tS, S$ and $\Ptwo$. The inverse map of $\phi_2$ is
\begin{equation*}
  \begin{array}[h]{cccc}
    \phi_3: & \Ptwo&\rto&S,\\
    &(\e_4':\al_1':\e_8')&\mapsto&(\e_4'^3:\e_4'^2\al_1':\e_4'^2\e_8':\e_8'\al_2')
  \end{array}
\end{equation*}
where $\al_2'=-\e_4'\e_8'-\al_1'^2$. The maps $\phi_2, \phi_3$ give a
bijection between the complement $U$ of the lines on $S$ and $\{(\e_4':
\al_1': \e_8') \in \Ptwo \mid \e_4', \e_8' \ne 0\}$, and furthermore, induces
a bijection between $U(\QQ)$ and the integral points
\begin{equation*}
  \{(\e_4, \al_1, \e_8,\al_2) \in \ZZp \times \ZZ \times \ZZnz \times \ZZ \mid
  \gcd(\e_4,\al_1,\al_2)=1,\ \al_2+\e_4\e_8+\al_1^2=0\}.
\end{equation*}

Motivated by the way the curves $E_5, E_7, E_3, E_1, E_2, E_6$ occur in
$\phi_1$ as the blow-ups of intersection points of $\phi_1(E_4), \phi_1(E_8),
\phi_1(A_1), \phi_1(A_2)$, one introduces the following further variables
\begin{equation*}
  \begin{array}[h]{lll}
    \e_5=\gcd(\e_4,\e_8),&\e_7=\gcd(\e_5,\e_8),& 
    \e_3=\gcd(\e_4, \al_1, \al_2),\\
    \e_1=\gcd(\e_3,\e_4,\al_2),&\e_2=\gcd(\e_1,\al_2)&\e_6=\gcd(\e_2,\al_2).
  \end{array}
\end{equation*}
Although we omit the details here, it is now straightforward to derive the
bijection described in the statement of Lemma \ref{lem:bijection} using
elementary number theory.

In analysing the height conditions apparent in \eqref{eq:height} we will meet
a number of real-valued functions, whose size it will be crucial to
understand.  We begin with the observation that \eqref{eq:height} is
equivalent to $h(\ee',\al_1;B) \le 1$,
where \[h(\ee',\al_1;B)=B^{-1}\max\left\{
\begin{aligned}
  &|\sbase 4 3 2 3 2 2\e_7| , |\sbase 3 2 2 2 1 1 \al_1|,\\
  & |\sbase 2 1 1 2 2 0 \e_7^2 \e_8|,
  \left|\frac{\e_4\e_5^2\e_7^3\e_8^2+\e_3\e_8\al_1^2}{\e_2\e_6^2}\right|
\end{aligned}
\right\}.\]
In what follows we will need to work with
the regions 
\[
\begin{split}
  \R(B) &= \{(\ee', \al_1) \in \RR^9 \midd \e_1, \dots, \e_7, |\e_8| \ge 1,\
  h(\ee', \al_1;B) \le 1\},\\
  \R'_1(B) &=\{\ee \in \RR^6 \mid \e_1, \dots, \e_6 \ge 1,\ \sbase 4 3 2 3 2
  2
  \le B,\ \sbase 6 5 3 4 2 4 \ge B\},\\
  \R'_2(\ee;B) &=\{(\e_7,\e_8,\al_1) \in \RR^3 \mid \e_7 \ge 0,\
  h(\ee',\al_1;B)
  \le 1\},\\
  \R'(B) &= \{(\ee',\al_1) \in \RR^9 \mid \ee \in \R'_1(B), (\e_7,\e_8,\al_1)
  \in \R'_2(\ee;B)\},\\
  &= \left\{(\ee', \al_1) \in \RR^9 \midd
    \begin{aligned}
      &\e_1, \dots, \e_6 \ge 1,\ \e_7 \ge 0,\ h(\ee', \al_1;B) \le 1,\\
      &\sbase 4 3 2 3 2 2 \le B,\ \sbase 6 5 3 4 2 4 \ge B
    \end{aligned}
\right\}.
\end{split}
\]
In keeping with the philosophy of \cite{arXiv:0710.1560}, the
definitions of these regions is dictated by the polytope whose volume
is defined to be the constant
$\alpha(\tS)$, as computed using an alternative method in the introduction. In fact one has
\begin{equation}\label{eq:alpha_volume}
\begin{split}
  \alpha(\tS) &= \vol\left\{\xx \in \RR_{\ge 0}^7 \midd
  \begin{aligned}
    &2x_1+2x_2+x_3+x_4+2x_6-x_7 \ge 0,\\
    &4x_1+3x_2+2x_3+3x_4+2x_5+2x_6+x_7=1
  \end{aligned}
  \right\}\\
  &= \vol\left\{\xx \in \RR_{\ge 0}^6 \midd
  \begin{aligned}
    &6x_1+5x_2+3x_3+4x_4+2x_5+4x_6 \ge 1,\\
    &4x_1+3x_2+2x_3+3x_4+2x_5+2x_6 \le 1
  \end{aligned}
  \right\},
\end{split}
\end{equation}
to which $\R'_1(B)$ is closely related.

Perhaps a few more words are in order concerning the role of the cone of
effective divisors in our work. The parametrisation of $U(\QQ)$ in Lemma
\ref{lem:bijection} suggests that $N_{U,H}(B)$ should be comparable to the
volume of $\R(B)$. On the other hand, the factors $\alpha(\tS)$ and
$\omega_\infty$ of the conjectured leading constant in our theorem suggest the
appearance of $\R'(B)$ instead.  The latter is constructed from $\R'_1(B)$,
which comes from the dual of the effective cone, and from 
$\R'_2(\ee;B)$, which is obtained from the region whose volume is
$\omega_\infty$.  At some point we will therefore need to make a transition
from $\R(B)$ to $\R'(B)$.  Rather than distributing this procedure over the
entire proof, as in our previous investigation \cite{arXiv:0710.1560}, we will
save this transition until Lemma~\ref{lem:final_step}, where it signifies the
final step in our argument.

We are now ready to record the various integrals that will feature in
our work, together with some basic estimates for them.
All of the bounds are simple enough to deduce in themselves,
but readily follow from applications of 
\cite[Lemma~5.1]{manin_auxiliary}.
Bearing this in mind, we have
\begin{align}
  V_1^a(\ee';B) &= \int_{(\ee', t_1) \in \R(B),\ \e_7 \ge
    |\e_8|}
  \frac{1}{\e_2\e_6^2} \dd t_1,\notag\\
  V_1^b(\ee';B) &= \int_{(\ee', t_1) \in \R(B),\ |\e_8| >
    \e_7}
  \frac{1}{\e_2\e_6^2} \dd t_1,\notag\\
  \label{eq:estimate_V1} 
V_1(\ee';B) &= \sum_{\iota\in\{a,b\}}V_1^\iota(\ee';B) 
 \ll
  \frac{B^{1/2}}{\e_2^{1/2}\e_3^{1/2}\e_6|\e_8|^{1/2}},
\end{align}
and
\begin{align}
  V_2^a(\ee,\e_8;B) &= \int_{t_7} V_1^a(\ee,t_7,\e_8;B) \dd t_7\notag\\ 
  \label{eq:estimate_V2a} &\ll
  \min\left\{\frac{B^{5/6}}{\sbase{0}{1/6}{1/2}{1/3}{2/3}{1/3}|\e_8|^{7/6}},
    \frac{B^2}{\sbase 7 6 4 5 3 5}\right\},\\
  \label{eq:estimate_V2b} V_2^b(\ee,\e_7;B) &= \int_{t_8}
  V_1^b(\ee,\e_7,t_8;B) \dd t_8 \ll
  \frac{B^{3/4}}{\sbase{0}{1/4}{1/2}{1/4}{1/2}{1/2}\e_7^{3/4}},
\end{align}
and finally
\begin{align}
  V_3^a(\ee;B)&=\int_{t_8} V_2^a(\ee,t_8;B) \dd t_8,\notag\\
  V_3^b(\ee;B)&=\int_{t_7} V_2^b(\ee,t_7;B) \dd t_7,\notag\\
  \label{eq:estimate_V3} V_3(\ee;B)&=V_3^a(\ee;B) + V_3^b(\ee;B) \ll
  \frac{B}{\sbase 1 1 1 1 1 1}.
\end{align}
We now have everything in place to start the proof of the theorem.

\section{First summation}

For fixed $\e_1, \dots, \e_8$, let $N_1$ be the number of $(\al_1, \al_2)$
that contribute to $N_{U,H}(B)$.  Let $I=I(\ee';B)$ be the set of $t_1 \in
\RR$ satisfying $h(\ee',t_1;B)\le 1$. By definition, $V_1(\ee';B) =
\meas(I)/(\e_2\e_6^2)$.  

We would like to begin by applying \cite[Proposition~2.4]{manin_auxiliary},
which is concerned with a much more general setting. In order to facilitate
our use of this result, Table~\ref{t:dict} presents a dictionary between the
notation adopted in \cite{manin_auxiliary} and the special case considered
here.

\begin{table}[!ht]
\begin{center}
\begin{tabular}{|r|l||r|l|}
\hline
$(r,s,t)$ & $(3,1,2)$& $ \delta$ & $\eta_1$\\
\hline 
($\alpha_0$;$\alpha_1,\ldots,\alpha_r$) & $(\eta_8;\eta_4,\eta_5,\eta_7)$ &
$(a_0;a_1,\ldots,a_r)$ & $(1;1,2,3)$ \\
\hline
($\beta_0$;$\beta_1,\ldots,\beta_s$) & $(\al_1;\eta_3)$ &
$(b_0;b_1,\ldots,b_s)$ & $(2;1)$ \\
\hline
($\gamma_0$;$\gamma_1,\ldots,\gamma_t$) & $(\al_2;\eta_2,\eta_6)$ &
$(c_1,\ldots,c_t)$ & $(1,2)$ \\
\hline
$\Pi(\aa)$ & $\eta_4\eta_5^2\eta_7^3$ & $\Pi'(\delta,\aa))$ & $\eta_1\eta_4\eta_5$\\
\hline
$\Pi(\bb)$ & $\eta_3$ & $\Pi'(\delta,\bb))$ & $\eta_1$\\
\hline
$\Pi(\gag)$ & $\eta_2\eta_6^2$ & $\Pi'(\delta,\gag))$ & $\eta_1\eta_2$\\
\hline
\end{tabular}
\end{center}
\caption{Dictionary for applying \cite[Proposition~2.4]{manin_auxiliary}}
\label{t:dict}
\end{table}

We may now apply \cite[Proposition~2.4]{manin_auxiliary} to deduce that
\[N_1 = \theta_1(\ee')V_1(\ee';B)+R_1(\ee';B),\]
where
$$
  \theta_1(\ee')=
\sum_{\substack{k|\e_1\e_2\\\cp{k}{\e_3\e_4}}} 
\frac{\mu(k)   
    \phis(\e_1\e_4\e_5\e_7)}{k\phis(\gcd(\e_1,k\e_2))} 
\sum_{\substack{1 \le \rho
      \le k\e_2\e_6^2\\\congr{\e_4\e_7\e_8}{-\rho^2\e_3}{k\e_2\e_6^2}\\\cp{\rho}{k\e_2\e_6^2}}}1
$$
and the error term $R_1(\ee';B)$ is the sum of terms of the form 
$$
\sum_{\substack{k_2|\e_1\e_2\\\cp{k_2}{\e_3\e_4}}} \mu(k_2)
\sum_{\substack{k_1|\e_1\e_4\e_5\e_7\\\cp{k_1}{k_2\e_2}}} \mu(k_1)
\sum_{\substack{1\le \rho \le k_2\e_2\e_6^2\\
    \congr{\e_4\e_7\e_8}{-\rho^2\e_3}{k_2\e_2\e_6^2}\\\cp{\rho}{k_2\e_2\e_6^2}}}
A,
$$
with
\[A=\psi\left( \frac{k_1^{-1}b_0-\rho\e_5\e_7\overline{k_1}}
  {k_2\e_2\e_6^2}\right) -
\psi\left(\frac{k_1^{-1}b_1-\rho\e_5\e_7\overline{k_1}}
  {k_2\e_2\e_6^2}\right),\] 
one for each of the intervals that form $I$,
with start and end points $b_0=b_0(\ee';B)$ and $b_1=b_1(\ee';B)$.
Here, $\overline{a}$ denotes the multiplicative inverse of an integer $a\in
(\ZZ/k_2\e_2\e_6^2\ZZ)^*$.  
Our first task is to show that the overall contribution from $R_1$ makes a
satisfactory contribution to $N_{U,H}(B)$.

\begin{lemma}\label{lem:d5_first_sum} We have
  \[N_{U,H}(B) = \sum_{\substack{\ee' \in \ZZp^7 \times
    \ZZnz\\\text{(\ref{eq:cpe8}), (\ref{eq:cpe7}), (\ref{eq:cpe})}}}
    \theta_1(\ee')V_1(\ee';B) + O(B(\log B)^5)\]
\end{lemma}

\begin{proof}
  We must show that once summed over $\ee' \in \ZZp^7 \times
    \ZZnz$ such that (\ref{eq:cpe8}), (\ref{eq:cpe7}) and
    (\ref{eq:cpe}) hold, the term $R_1(\ee';B)$ contributes $O(B(\log
  B)^5)$.
    Let $q=k_2\e_2\e_6^2$. 
  We remove (\ref{eq:cpe8}) by a M\"obius inversion.  This leads us to
    estimate 
  \begin{equation*}
    \sum_{\substack{(\ee,\e_7) \in \ZZp^7 \\\text{(\ref{eq:cpe7}),
          (\ref{eq:cpe})}}} R_1'(\ee,\e_7;B),
  \end{equation*}
  where $R_1'(\ee,\e_7;B)$ is defined to be
  \begin{equation*}
    \sum_{\substack{k_2|\e_1\e_2\\
        \cp{k_2}{\e_3\e_4}}} \mu(k_2)
    \sum_{\substack{k_1|\e_1\e_4\e_5\e_7\\\cp{k_1}{k_2\e_2}}} \mu(k_1)
    \sum_{k_8|\e_1\e_2\e_3\e_4\e_5\e_6} \mu(k_8) A',
  \end{equation*}
  with
  $$
  A'=\sum_{\e_8' \in \ZZ \cap I'} 
  \sum_{\substack{1\le \rho \le q\\
      \congr{\e_4\e_7 k_8\e_8'}{-\rho^2\e_3}{q}\\\cp{\rho}{q}}}
  \sum_{i \in \{0,1\}} (-1)^i
  \psi\left( \frac{k_1^{-1}b_i-\rho\e_5\e_7\overline{k_1}}{q}\right),
  $$
  where $I'$ is the allowed interval for $\e_8'$ and 
  $b_0,b_1$ as above depend on $\e_1, \dots, \e_7$ and $\e_8=k_8\e_8'$.
  We may split the summation over $\e_8' \in I'$ into subintervals $I''$ where
  we have 
  $b_0,b_1 \in C^1(I'',\lambda_0)$ 
  as functions of $\e_8'$.  In view of the bounds for $|k_8\e_8'|$ and
  $|k_1\al_1'|$ 
  that follow from the inequalities in the definition of
  $\mathcal{R}(B)$, it follows that
  $$
  \gothm(I'')\ll 
  \frac{B}{\sbase 211220 \e_7^2}, \quad 
  \lambda_0 \ll \frac{B}{\sbase 322211}.
  $$

  Since $\cp{\e_3}{q}$, we may restrict the summation over
  $k_8$ to $k_8 \mid \e_1\e_3\e_4\e_5$ such that
  $\cp{k_8}{q}$. Then $\cp{\e_3\e_4\e_7k_8}{q}$
  and 
  $\cp{\e_5\e_7\overline{k_1}}{q}$, 
  so that we can
  apply Lemma~\ref{lem:lemma10} to obtain
  $$
    A' \ll \left(h_2(q) q^{1/2}+ \frac{\tau(q)^2B}{\sbase
      221222 \e_7^2} + \frac{h_1(q)B}{(\log B)\sbase
      {5/2}{7/4}{3/2}{2}{3/2}{1}\e_7}\right)(\log B)^2.
  $$

  Note that $h_k(\eta_2\eta_6^2)\ll_\ve \eta_2^\ve h_{2k}(\eta_6)$ for any
  $k\in \ZZp$. Writing, temporarily, $\mathcal{L}=\log B$ we deduce that the
  total contribution from the first term is
  \begin{equation*}
    \begin{split}
      \sum_{\e_1, \dots, \e_7} \sum_{k_1,k_2,k_8} h_2(q)\mathcal{L}^2q^{1/2}
      &\ll_\ve \sum_{\e_1, \dots, \e_7} (\e_1\e_2\e_3\e_4\e_5)^\ve
      2^{\omega(\e_7)} h_4(\e_6) \mathcal{L}^2
      \e_1^{1/2} \e_2\e_6\\
      &\ll_\ve  \sum_{\e_1, \dots, \e_6} (\e_1\e_2\e_3\e_4\e_5)^\ve h_4(\e_6) 
      \mathcal{L}^3 \frac{B}{\sbase {7/2} 2 2 3 2 1}\\
      &\ll_\ve B \mathcal{L}^5,
    \end{split}
  \end{equation*}
  by Lemma~\ref{lem:sum_h_k}.
  The total contribution from the second term is
  \begin{equation*}
    \begin{split}
      \sum_{\e_1, \dots, \e_7} \sum_{k_1,k_2,k_8} \tau(q)^2 \mathcal{L}^2
      \frac{B}{\sbase 2 2 1 2 2 2 \e_7^2}&
      \ll B \mathcal{L}^2 \sum_{\e_3} \frac{ 2^{\omega(\e_3)}}{\e_3}\\
      &\ll B\mathcal{L}^4.
    \end{split}
  \end{equation*}
  Finally, the total contribution from the third term is
  \begin{equation*}
    \begin{split}
      \sum_{\e_1, \dots, \e_7} \sum_{k_1,k_2,k_8}
      \frac{h_1(q)B\mathcal{L}}{\sbase {5/2}{7/4}{3/2}{2}{3/2}{1}\e_7}
      &\ll_\ve B \mathcal{L}\sum_{\e_1, \dots, \e_7}
      \frac{(\e_1\e_2\e_3\e_4\e_5)^\ve 2^{\omega(\e_7)}
        h_2(\e_6)}{\sbase{5/2}{7/4}{3/2}{2}{3/2}{1} \e_7}\\
      &\ll_\ve B\mathcal{L}^{5}.
    \end{split}
  \end{equation*}
  This therefore completes the proof of the lemma.  
\end{proof}

\section{Second summation}

Let $N_{U,H}^a(B)$ be the number of $(\ee', \aa) \in \TT(B)$ subject
to $|\e_8| \le \e_7$, and let $N_{U,H}^b(B)$ be the remaining number of 
elements of $\TT(B)$. Lemma~\ref{lem:d5_first_sum} can be modified
in an obvious way to give estimates for $N_{U,H}^a(B)$ and
$N_{U,H}^b(B)$. For $N_{U,H}^a(B)$, we sum over $\e_7$ first and over
$\e_8$ afterwards, and for $N_{U,H}^b(B)$, we do the reverse.

\subsection{Case $|\e_8| > \e_7$}

We rewrite the result of Lemma~\ref{lem:d5_first_sum} as follows. Removing
(\ref{eq:cpe8}) by a M\"obius inversion, and adding $\cp{k_8}{k\e_2\e_6^2}$ to
prevent that $A=0$, we arrive at the formula
\begin{multline*}
  N_{U,H}^b(B) = \sum_{\substack{(\ee,\e_7) \in \ZZp^7\\\text{(\ref{eq:cpe7}),
        (\ref{eq:cpe})}}} \sum_{\substack{k\mid
      \e_1\e_2\\\cp{k}{\e_3\e_4}}} \frac{\mu(k)
    \phis(\e_1\e_4\e_5\e_7)}{k\phis(\gcd(\e_1,k\e_2))}\\\times
  \sum_{\substack{1 \le \rho \le k\e_2\e_6^2\\\cp{\rho}{k\e_2\e_6^2}}}
  \sum_{\substack{k_8|\e_1\e_3\e_4\e_5\\\cp{k_8}{k\e_2\e_6^2}}} \mu(k_8) A +
  O(B(\log B)^5),
\end{multline*}
where \[A = \sum_{\substack{\e_8' \in
    \ZZnz\\\congr{\e_4\e_7k_8\e_8'}{-\rho^2\e_3}{k\e_2\e_6^2}\\k_8|\e_8'|>
    \e_7}} V_1^b(\ee,\e_7,k_8\e_8';B).\]

\begin{lemma}\label{lem:8}
  We have
  \[N_{U,H}^b(B) = \sum_{\substack{(\ee,\e_7) \in
      \ZZp^7\\\text{(\ref{eq:cpe7}), (\ref{eq:cpe})}}} \theta_2^b(\ee,\e_7)
      V_2^b(\ee,\e_7;B)+O(B(\log B)^5),\]
      where \[\theta_2^b(\ee,\e_7)=\phis(\e_1\e_2\e_3\e_4\e_5\e_6)
      \phis(\e_1\e_2\e_4\e_5\e_7) \prod_{\substack{p\mid \e_1\\p \nmid
          \e_2\e_3\e_4}} \frac{1-2/p}{1-1/p}.\]
\end{lemma}

\begin{proof}
  Let $q=k\e_2\e_6^2$ and 
  \begin{multline*} N(t_1,t_2) = 
    \sum_{\substack{k\mid \e_1\e_2\\\cp{k}{\e_3\e_4}}} \frac{\mu(k)
      \phis(\e_1\e_4\e_5\e_7)}{k\phis(\gcd(\e_1,k\e_2))}\\\times
    \sum_{\substack{k_8|\e_1\e_3\e_4\e_5\\\cp{k_8}{q}}}\mu(k_8)
    \sum_{\substack{1 \le \rho \le q\\\cp{\rho}{q}}} N'_{k,k_8}(\rho;t_1,t_2)
  \end{multline*}
  where \[N'_{k,k_8}(\rho;t_1,t_2)=\{\e_8' \in (t_1/k_8,t_2/k_8] \mid
  \congr{\e_4\e_7k_8\e_8'}{-\rho^2\e_3}{q}\}.\] 

  As in \cite[Section~8.3]{MR2332351}, we have \[N'_{k,k_8}(\rho;t_1,t_2) =
  \frac{t_2-t_1}{k_8q} + \psi\left(\frac{k_8^{-1}t_1-a\rho^2}{q}\right)-
  \psi\left(\frac{k_8^{-1}t_2-a\rho^2}{q}\right)\] where $a$ is the unique
  integer modulo $q$ with  $\congr{\e_4\e_7k_8a}{-\e_3}{q}.$ Clearly 
   $\congr{\e_4\e_7k_8\e_8'}{-\rho^2\e_3}{q}$ is
  equivalent to $\congr{\e_8'}{a\rho^2}{q}$ for any such $a$.
  Using Lemma~\ref{lem:sum_mod_q}, we deduce that $N(t_1,t_2)$ is
  $$
   (t_2-t_1)\theta_2^b(\ee,\e_7) + O\left(2^{\omega(\e_1\e_2)}2^{\omega(\e_1\e_3\e_4\e_5)}h_1(\e_2\e_6^2)
      \log(\e_2\e_6^2+1) (\e_2\e_6^2)^{1/2}\right).
$$
A straightforward application of partial summation therefore reveals
the total error as being 
\[
  \begin{split}
    &\ll \sum_{\ee,\e_7}
    2^{\omega(\e_1\e_2)}2^{\omega(\e_1\e_3\e_4\e_5)}h_1(\e_2\e_6^2)(\log
    B) (\e_2\e_6^2)^{1/2} \sup_{|\e_8|>\e_7} V_1^b(\ee,\e_7,\e_8;B)\\
    &\ll \sum_{\ee}
    2^{\omega(\e_1\e_2)}2^{\omega(\e_1\e_3\e_4\e_5)}h_1(\e_2\e_6^2)\frac{B\log
      B}{\sbase
      2{3/2}{3/2}{3/2}11}\\
    &\ll B(\log B)^5. 
  \end{split}
  \]
  Here, in the second step, we have used $\sbase 4 3 2 3 2 2 \e_7 \le B$ and
  $|\e_8|>\e_7$ and the bound (\ref{eq:estimate_V1}) for $V_1^b$. The
  final step uses Lemma \ref{lem:sum_h_k}.
\end{proof}

\subsection{Case $\e_7 \ge |\e_8|$}

We rewrite the result of Lemma~\ref{lem:d5_first_sum}.  Recall the definition
\eqref{eq:def_f} of the function $f_{a,b}$ for positive integers $a,b$.
Noting that we may replace (\ref{eq:cpe7}) by $\cp{\e_7}{\e_1\e_3\e_4}$, it
follows that
\begin{multline*}
  N_{U,H}^a(B) = \sum_{\substack{(\ee,\e_8) \in \ZZp^6 \times
      \ZZnz\\\text{(\ref{eq:cpe8}), (\ref{eq:cpe})}}} \sum_{\substack{k\mid
      \e_1\e_2\\\cp{k}{\e_3\e_4}}} \frac{\mu(k)
    \phis(\e_1\e_4\e_5)}{k\phis(\gcd(\e_1,k\e_2))}\\\times \sum_{\substack{1
      \le \rho \le k\e_2\e_6^2\\\cp{\rho}{k\e_2\e_6^2}}} A + O(B(\log B)^5)
\end{multline*}
where \[A = \sum_{\substack{\e_7 \in
    \ZZnz\\\congr{\e_4\e_7\e_8}{-\rho^2\e_3}{k\e_2\e_6^2}\\\e_7\ge|\e_8|}}
f_{\e_5,\e_1\e_3\e_4}(\e_7) V_1^a(\ee,\e_7,\e_8;B).
\]
Here we automatically have $\cp{\e_4\e_8}{k\e_2\e_6^2}$. Thus the congruence
involving $\rho$ in $A$ determines $\e_7$ uniquely
modulo $k\e_2\e_6^2$.

\begin{lemma}\label{lem:9}
  We have
  \[N_{U,H}^a(B) = \sum_{\substack{(\ee,\e_8) \in \ZZp^6\times
      \ZZnz\\\text{(\ref{eq:cpe8}), (\ref{eq:cpe})}}} \theta_2^a(\ee,\e_8)
  V_2^a(\ee,\e_8;B)+O(B(\log B)^5),\] where 
\[\theta_2^a(\ee,\e_7)=
\phis(\e_1\e_2\e_3\e_4\e_6) \phis(\e_1\e_2\e_4\e_5)
  \prod_{\substack{p\mid \e_1\\p\nmid \e_2\e_3\e_4}} \frac{1-2/p}{1-1/p}
  \prod_{p \nmid \e_1\e_2\e_3\e_4\e_5\e_6} \hspace{-0.3cm}(1-1/p^2).\]
\end{lemma}

\begin{proof}
  Let $q=k\e_2\e_6^2$, $q'=\e_2\e_6^2$ and 
  \begin{multline*}
    N(t_1,t_2) = \sum_{\substack{k\mid \e_1\e_2\\\cp{k}{\e_3\e_4}}}
    \frac{\mu(k) \phis(\e_1\e_4\e_5)}{k\phis(\gcd(\e_1,k\e_2))}
    \\\times \sum_{\substack{1 \le \rho \le q\\\cp{\rho}{q}}} \sum_{\substack{t_1 <
        \e_7 \le t_2\\\congr{\e_4\e_7\e_8}{-\rho^2\e_3}{q}}}
    f_{\e_5,\e_1\e_3\e_4}(\e_7).
  \end{multline*}
  It  follows from Lemma~\ref{lem:sum_mod_q_2} that $N(t_1,t_2)$ is
  \begin{multline*}
    (t_2-t_1) 
    \sum_{\substack{k\mid \e_1\e_2\\\cp{k}{\e_3\e_4}}}
    \frac{\mu(k) \phis(\e_1\e_4\e_5)}{k\phis(\gcd(\e_1,k\e_2))}\cdot
    \phis(\e_1\e_3\e_4q) \prod_{p \nmid q\e_1\e_3\e_4\e_5} (1-1/p^2)\\
    + O\left(2^{\omega(\e_1\e_2)}2^{\omega(\e_1\e_3\e_4)}(\log
    B)h_1(q')\log(q'+1)q'^{1/2}\right).
  \end{multline*}

  A little thought reveals that the main term here is
  $(t_2-t_1)\theta_2^a(\ee, \e_7)$. Using partial summation, we estimate the
  total error as \[
  \begin{split}
    &\ll \sum_{\ee,\e_8} 2^{\omega(\e_1\e_2)}2^{\omega(\e_1\e_3\e_4)} (\log
    B)^2 h_1(\e_2\e_6^2)(\e_2\e_6^2)^{1/2}
    \sup_{\e_7\ge|\e_8|} V_1^a(\ee,\e_7,\e_8;B)\\
    &\ll B(\log B)^2 \sum_{\ee}
    \frac{2^{\omega(\e_1\e_2)}2^{\omega(\e_1\e_3\e_4)}
      h_1(\e_2\e_6^2)}{\sbase 2{3/2}{3/2}{3/2}11}\\
    &\ll B(\log B)^5,
  \end{split}
  \]
  using $\sbase 432322|\e_8| \le \sbase 4 3 2 3 2 2 \e_7 \le B$ and
  (\ref{eq:estimate_V1}) in the second step and Lemma \ref{lem:sum_mod_q} in the
  final step.
\end{proof}

\section{Third summation}

Throughout the remainder of the paper we set $E=B(\log B)^5(\log \log B)$ for
the total error term that appears in our main result.  In this section and the
next we will need to compute the average order of certain complicated
multi-variable arithmetic functions, sometimes weighted by piecewise
continuous functions.  As previously, we will place ourselves in the more
general investigation carried out in \cite{manin_auxiliary}. Here, given $r\in
\ZZp$ and $C\in \RR_{\geq 1}$, a number of rather general sets of functions
are introduced: $\Theta_{1,r}(C,\e_r)$ \cite[Definition~3.8]{manin_auxiliary},
$\Theta_{2,r}(C)$ \cite[Definition~4.2]{manin_auxiliary}, $\Theta_{3,r}'$
\cite[Definition~7.7]{manin_auxiliary} and $\Theta_{4,r}'(C)$
\cite[Definition~7.8]{manin_auxiliary}.  We will not redefine these sets here,
but content ourselves with recording the inclusions
$$
\Theta_{3,r}' \supset \Theta_{4,r}'(C) \subset \Theta_{1,r}(48rC^2,\e_r) \cap
\Theta_{2,r}(48r(3^rC)^2)
$$
\cite[Corollary~7.9]{manin_auxiliary}.

In the notation of \cite[Definition~7.7]{manin_auxiliary}, our
manipulations will involve the function
\begin{equation}\label{eq:t3}
  \theta_3(\ee) = 
  \prod_p \theta_{3,p}(I_p(\ee)) \in \Theta_{3,6}'
\end{equation}
for any $\ee\in \ZZp^6$, where $I_p(\ee)=\{i\in \{1,\ldots,6\}:p\mid \eta_i\}$
and
\begin{equation*}
  \theta_{3,p}(I)=
  \begin{cases}
    (1-\frac 1 {p^2}), & I= \emptyset,\\
    (1-\frac 1 p)^2(1-\frac 2 p), &I = \{1\},\\
    (1-\frac 1 p)^3, &I =
    \{2\},\{4\},\{1,2\},\{1,3\},\{1,4\},\{2,6\},\{4,5\},\\
    (1-\frac 1 p)^2, &I = \{3\},\{5\},\{6\},\\
    0, &\text{all other $I \subset \{1, \dots, 6\}$.}
  \end{cases}
\end{equation*}

\subsection{Case $|\e_8| > \e_7$}

\begin{lemma}\label{lem:d5_third_sum_b}
  We have
  \begin{equation*}
    N_{U,H}^b(B) = 
    \sum_{\ee \in \ZZp^6}
    \theta_3(\ee) V_3^b(\ee;B) + O(E)
  \end{equation*}
  where $\theta_3$ is given by \eqref{eq:t3}.
\end{lemma}

\begin{proof}
  Our proof of the lemma is based on combining
  \cite[Proposition~3.9]{manin_auxiliary} with Lemma \ref{lem:8}.  We will
  apply the former to $\theta(\ee,\e_7)V_2^b(\ee,\e_7;B)$ summed over $\e_7
  \ge 1$, with $(r,s)=(5,1)$ and
  \begin{equation*}
    \theta(\ee,\e_7)=
    \begin{cases}
      \theta_2^b(\ee,\e_7), &\text{if (\ref{eq:cpe7}), (\ref{eq:cpe}) hold,}\\
      0, &\text{otherwise.}
    \end{cases}
  \end{equation*}
  
  There are a number of preliminary hypotheses that need to be checked in
  using \cite[Proposition~3.9]{manin_auxiliary}.  Local factors of
  $\theta=\prod_p \theta_p(I_p(\ee,\e_7)) \in \Theta_{3,7}'$ are given by
  $\theta_p(I)$, equal to
  \begin{equation*}
    \begin{cases}
      1, & I= \emptyset,\\
      (1-\frac 1 p)(1-\frac 2 p), &I = \{1\},\\
      (1-\frac 1 p)^2, &I =
      \{2\},\{4\},\{5\},\{1,2\},\{1,3\},\{1,4\},\{2,6\},\{4,5\},\{5,7\}\\
      1-\frac 1 p, &I = \{3\},\{6\},\{7\},\\
      0, &\text{all other $I \subset \{1, \dots, 7\}$.}
    \end{cases}
  \end{equation*}
  We see that $\theta \in \Theta_{4,7}'(3)\subset
  \Theta_{1,7}(C, \e_7)$, for an
  appropriate $C \in \ZZp$.

  For $V_2^b$, we observe that (\ref{eq:estimate_V2b})
  implies \[V_2^b(\ee,\e_7;B) \ll \frac{B}{\sbase 111111 \e_7} \cdot
  \left(\frac{B}{\sbase 432322 \e_7}\right)^{-1/4}\] and that
  $V_2^b(\ee,\e_7;B) = 0$ unless $\sbase 432322 \e_7 \le B$.

  Thus everything is in place for an application of
  \cite[Proposition~3.9]{manin_auxiliary}, giving
  \begin{align*}
    \sum_{\e_7\ge 1}\theta(\ee,\e_7)V_2^b(\ee,\e_7;B) &=
    \mathcal{A}(\theta(\ee,\e_7),\e_7)
    \int_1^B V_2^b(\ee,t_7;B) \dd t_7+O(E)\\
    &= \theta_3(\ee) V_3^b(\ee;B)+O(E)
    ,
  \end{align*} 
  where we check $\mathcal{A}(\theta(\ee,\e_7),\e_7) = \theta_3(\ee)$ by
  \cite[Corollary~7.10]{manin_auxiliary}.
\end{proof}

\subsection{Case $\e_7 \ge |\e_8|$}

\begin{lemma}\label{lem:d5_third_sum_a}
  We have
  \begin{equation*}
    N_{U,H}^a(B) = 
    \sum_{\ee \in \ZZp^6}
    \theta_3(\ee) V_3^a(\ee;B) + O(E)
  \end{equation*}
  with $\theta_3$ given by \eqref{eq:t3}.
\end{lemma}

\begin{proof}
  This time our argument is based on combining
  \cite[Proposition~3.10]{manin_auxiliary} with Lemma \ref{lem:9}, the former
  being applied with $r=5$. As previously, there are a number of preliminary
  hypotheses that need to be checked in order to use this result. For the
  first of these, we define
  \begin{equation*}
    \theta(\ee,\e_8)=
    \begin{cases}
      \theta_2^a(\ee,\e_8), &\text{if (\ref{eq:cpe8}), (\ref{eq:cpe}) hold,}\\
      0, &\text{otherwise.}
    \end{cases}
  \end{equation*}
  As in the proof of Lemma~\ref{lem:d5_third_sum_b}, we have $\theta \in
  \Theta_{1,7}(C,\e_8)$,  for some $C \in \ZZp$.
  
  Next, (\ref{eq:estimate_V2a}) implies that
  \begin{multline*}
    V_2^a(\ee,\e_8;B)\\ \ll \frac{B}{\sbase 1 1 1 1 1 1 |\e_8|}
    \min\left\{\left(\frac{B}{\sbase 6 5 3 4 2 4
          |\e_8|^{-1}}\right)^{-1/6},\frac{B}{\sbase 6 5 3 4 2 4
        |\e_8|^{-1}}\right\}.
  \end{multline*}
  An application of \cite[Proposition~3.10]{manin_auxiliary} now gives
  the expected main term, together with a total error term
  $O(E)$. 
\end{proof}

\section{Completion of the proof}

We put back together our estimates for $N_{U,H}^b(B)$ and $N_{U,H}^a(B)$ that
were obtained in Lemmas~\ref{lem:d5_third_sum_b} and \ref{lem:d5_third_sum_a},
respectively. This yields the following result.

\begin{lemma}
  We have \[N_{U,H}(B) = 
  \sum_{\ee \in \ZZp^6}
  \theta_3(\ee) V_3(\ee;B) + O(E),\] with
  $\theta_3$ given by \eqref{eq:t3}.
\end{lemma}

It remains to handle the summation over $\e_1, \dots, \e_6$. This is achieved
in the next result.

\begin{lemma}
  \[N_{U,H}(B) = \left(\prod_p \omega_p\right) \int_{(\ee',\al_1) \in \R(B)}
  \frac{1}{\e_2\e_6^2} \dd \ee' \dd \al_1 + O(E).\]
\end{lemma}

\begin{proof}
  Since $\theta_3 \in \Theta_{4,r}'(4)$, there is a $C \in \ZZp$ such that
  $\theta_3 \in \Theta_{2,r}(C)$.
  This and the bound (\ref{eq:estimate_V3}) for
  $V_3(\ee;B)$ show that we are able to apply
  \cite[Proposition~4.3]{manin_auxiliary} with $(r,s)=(6,0)$ to conclude that
  \begin{equation*}
    \sum_{\ee} 
    \theta_3(\ee)
    V_3(\ee;B) = \theta_0 V_0(B) + O(E).
  \end{equation*}
  Here,
  \begin{equation*}
    V_0(B) = \int_{\ee} V_3(\ee;B) \dd \ee = \int_{\R(B)} \frac{1}{\e_2\e_6^2}
    \dd \ee' \dd \al_1
  \end{equation*}
  and $\theta_0$ is the ``average'' of $\theta(\ee)$ over $\e_1, \dots,
  \e_6$, which is computed 
  as
  \begin{equation*}
    \begin{split}
      \theta_0 &= \prod_p \left(1-\frac 1 p\right)^7 \left(\left(1+\frac 1
          p\right) + \left(\frac 1 p - \frac 2 {p^2}\right) + 2\left(\frac 1 p
          - \frac 1{p^2}\right)+3 \frac 1 p+5 \frac 1 {p^2}\right)\\ & =
      \prod_p \left(1-\frac 1 p\right)^7 \left(1 + \frac 7 p + \frac
        1{p^2}\right)\\
      &= \prod_p \omega_p
    \end{split}
  \end{equation*}
  using \cite[Corollary~7.10]{manin_auxiliary}.
\end{proof}

The subsequent task is to modify the domain of integration, replacing $\R(B)$
by $\R'(B)$. This is the final step needed to extract the main term as it
appears in the statement of the theorem.

\begin{lemma}\label{lem:final_step}
  \[N_{U,H}(B) = \left(\prod_p \omega_p\right) \int_{(\ee',\al_1) \in \R'(B)}
  \frac{1}{\e_2\e_6^2} \dd \ee' \dd \al_1 + O(E).\]
\end{lemma}

\begin{proof}
  Let 
\[
V^{(i)}(B)=\int_{h(\ee',\al_1;B) \le 1,\ (\ee',\al_1) \in \R_i(B)} 
  (\e_2\e_6^2)^{-1} \dd \ee' \dd \al_1,
\]
  where
  \begin{equation*}
    \begin{split}
      \R_0(B)&=\{(\ee',\al_1) \in \RR^9 \mid \e_1, \dots, \e_7,|\e_8|\ge 1\}\\
      \R_1(B)&=\{(\ee',\al_1) \in \RR^9 \mid \e_1, \dots, \e_7,|\e_8|\ge 1,\
      \sbase 4 3 2 3 2 2 \le B\}\\
      \R_2(B)&=\left\{(\ee',\al_1) \in \RR^9 \midd
        \begin{aligned}
          &\e_1, \dots, \e_7,|\e_8|\ge 1,\\ 
          &\sbase 4 3 2 3 2 2 \le B,\ \sbase 6 5 3 4 2 4 \ge B
        \end{aligned}
      \right\}\\
      \R_3(B)&=\left\{(\ee',\al_1) \in \RR^9 \midd
        \begin{aligned}
          &\e_1, \dots, \e_7\ge 1,\\ 
          &\sbase 4 3 2 3 2 2 \le B,\ \sbase 6 5 3 4 2 4 \ge B
        \end{aligned}
      \right\}\\
      \R_4(B)&=\left\{(\ee',\al_1) \in \RR^9 \midd
        \begin{aligned}
          &\e_1, \dots, \e_6 \ge 1,\ \e_7 \ge 0,\\ 
          &\sbase 4 3 2 3 2 2 \le B,\ \sbase 6 5 3 4 2 4 \ge B
        \end{aligned}
      \right\}\\
    \end{split}
  \end{equation*}
  For $1\le i \le 4$, we will show that $|V^{(i)}(B)-V^{(i-1)}(B)| \ll
  B(\log B)^5$. Since $V^{(0)}(B)=V_0(B)$ and $V^{(4)}(B) = \int_{(\ee',\al_1)
    \in \R'(B)} (\e_2\e_6^2)^{-1} \dd \ee' \dd \al_1$, this is enough to
  establish the lemma.

  It turns out that in applying 
    \cite[Lemma~5.1]{manin_auxiliary}
    to obtain
  \eqref{eq:estimate_V1}--\eqref{eq:estimate_V3}, only the inequality
  $h(\ee',\al_1;B)\le 1$ is used in the definition of $\R(B)$. Hence the same
  bounds hold if we replace $\R(B)$ by $\R'(B)$ in the definitions of $V_i^a,
  V_i^b$.

  For $i=1$,  the inequality $\sbase 4 3 2 3 2 2 \le B$
  follows from $h(\ee',\al_1;B)\le 1$ and $\e_7 \ge 1$. Therefore,
  $V^{(0)}(B)=V^{(1)}(B)$. 

  For $i=2$ we use a variation of (\ref{eq:estimate_V2a}) for the integration
  over $\al_1, \e_7$. Then integrating over 
   $|\e_8| \ge 1$
    and $\sbase 6 5 3 4 2 4 < B$ and $1 \le \e_1, \dots, \e_5 \le B$, we deduce
  that
  \begin{equation*}
    \begin{split}
      |V^{(2)}(B)-V^{(1)}(B)| &\ll \int_{\ee, \e_8} \frac{B^{5/6}}{\sbase
        0{1/6}{1/2}{1/3}{2/3}{1/3}|\e_8|^{7/6}} \dd \ee \dd \e_8\\
      &\ll \int_{\e_1, \dots, \e_5} \frac{B}{\sbase 111110} \dd \e_1
      \dots\dd\e_5\\
      &\ll B(\log B)^5.
    \end{split}
  \end{equation*}

  For $i=3$ we begin by using 
  (\ref{eq:estimate_V1}) for the integration over $\al_1$. Then
  integrating over $|\e_8| <
  1$, $\e_7 \le B/(\sbase 432322)$, $\sbase 653424 \ge B$ and $1 \le \e_1,
  \dots, \e_5 \le B$, we deduce that  \begin{equation*}
    \begin{split}
      |V^{(3)}(B)-V^{(2)}(B)| &\ll \int_{\ee, \e_7,\e_8}
      \frac{B^{1/2}}{\e_2^{1/2}\e_3^{1/2}\e_6|\e_8|^{1/2}} \dd \ee \dd \e_7
      \dd \e_8\\
      &\ll \int_{\ee} \frac{B^{3/2}}{\sbase 4{7/2}{5/2}323} \dd \ee\\
      &\ll B(\log B)^5.
    \end{split}
  \end{equation*}

  Finally, for $i=4$ we use 
  (\ref{eq:estimate_V2b}) for the integration over $\al_2,\e_8$. Then,
  integrating over $0
  \le \e_7 < 1$, $\sbase 432322 \le B$ and $1 \le \e_1, \dots, \e_5
  \le B$ we obtain
  \begin{equation*}
    \begin{split}
      |V^{(4)}(B)-V^{(3)}(B)| &\ll \int_{\ee, \e_7} \frac{B^{3/4}}{\sbase
        0{1/4}{1/2}{1/4}{1/2}{1/2} \e_7^{3/4}} \dd \ee \dd \e_7\\
      &\ll \int_{\e_1, \dots, \e_5} \frac{B}{\sbase 111110} \dd \e_1 \dots \dd
      \e_5\\
      &\ll B(\log B)^5.
    \end{split}
  \end{equation*}
This completes the proof of the lemma. 
\end{proof}

Substituting \[x_0=\frac{\sbase 4 3 2 3 2 2 \e_7}{B},\quad
x_1 = \frac{\sbase 3 2 2 2 1 1 \al_1}{B},\quad x_2 = \frac{\sbase 2 1 1 2 2 0
  \e_7^2\e_8}{B}\] into $\omega_\infty$, for fixed $\ee \in
\RRp^6$, we obtain 
\[\int_{(\e_7,\e_8,\al_1) \in \R_2'(\ee;B)} \frac{1}{\e_2\e_6^2} \dd
\e_7 \dd \e_8 \dd \al_1 = \frac{\omega_\infty B}{\e_1\cdots \e_6}.\] Finally,
by substituting $x_i = \frac{\log \e_1}{\log B}$ for $1\le i \le 6$ into
(\ref{eq:alpha_volume}), written as an integral, we deduce that \[\alpha(\tS)
(\log B)^6 = \int_{\ee \in \R'_1(B)} \frac{1}{\e_1\cdots\e_6} \dd \ee.\] This
completes the proof of the theorem.


\begin{thebibliography}{99}

\bibitem{MR1620682} V.~V. Batyrev and Yu. Tschinkel, Manin's conjecture for
   toric varieties.
\emph{J. Alg. Geom.} {\bf 7} (1998),  15--53.

\bibitem{MR2000b:11074} R. de la Bret\`eche,
Sur le nombre de points de hauteur born\'ee d'une certaine
surface cubique singuli\`ere. {\em Ast\'erisque} {\bf 251}
(1998), 51--77.



\bibitem{MR2320172}
R.~de~la Bret{\`e}che and T.~D. Browning,
On {M}anin's conjecture for singular del {P}ezzo surfaces of degree
  4. {I}. {\em Michigan Math. J.} {\bf 55} (2007), 51--80.
  

\bibitem{MR2332351}
R.~de~la Bret{\`e}che, T.~D. Browning, and U.~Derenthal,
On {M}anin's conjecture for a certain singular cubic surface.
{\em Ann. Sci. \'Ecole Norm. Sup.}  {\bf 40} (2007), 1--50.


\bibitem{gauss} T.~D. Browning,
An overview of Manin's conjecture for del Pezzo surfaces.
{\em Analytic number theory --- A tribute to Gauss and Dirichlet
  (G\"ottingen, 20th June -- 24th June, 2005)}, 39--56, Clay Mathematics
  Proceedings {\bf 7}, AMS, 2007. 

\bibitem{arXiv:0710.1560}
T.~D. Browning and U.~Derenthal,
Manin's conjecture for a quartic del Pezzo surface with
  $\mathbf{A}_4$ singularity, arXiv:0710.1560, 2007.

\bibitem{clt}
A. Chambert-Loir and Yu. Tschinkel,
On the distribution of points of bounded height on equivariant
              compactifications of vector groups.
{\em Invent. Math.} {\bf 148} (2002), 421--452.

\bibitem{math.AG/0604194}
U.~Derenthal,
Singular Del Pezzo surfaces whose universal torsors are
  hypersurfaces, arXiv:math.AG/0604194, 2006.

\bibitem{MR2318651}
U.~Derenthal,
On a constant arising in {M}anin's conjecture for del {P}ezzo
  surfaces.
{\em Math. Res. Lett.}  {\bf 14} (2007), 481--489.

\bibitem{manin_auxiliary}
U.~Derenthal,
Counting integral points on universal torsors,
arXiv:0810.4122, 2008.


\bibitem{math.NT/0604193}    
U. Derenthal and Yu. Tschinkel,
Universal torsors over Del Pezzo surfaces and rational points.
{\em Equidistribution in Number Theory, An Introduction}, 169--196, 
NATO Sci. Ser. II Math. Phys. Chem. {\bf 237},
Springer, 2006.

\bibitem{math.AG/0703202}
U.~Derenthal, M.~Joyce and Z.~Teitler,
The nef cone volume of generalized Del Pezzo surfaces.
{\em Algebra \& Number Theory} {\bf 2} (2008), no. 2, 157--182. 


\bibitem{MR89m:11060}
J.~Franke, Yu~.I. Manin and Yu. Tschinkel,
Rational points of bounded height on {F}ano varieties.
{\em Invent. Math.} {\bf 95} (1989), 421--435.


\bibitem{MR1340296}
E. Peyre,
Hauteurs et nombres de Tamagawa sur les
vari\'et\'es de Fano.
{\em Duke Math. J.}  {\bf 79} (1995),
101--218.

\end{thebibliography}
\end{document}